%% file: main_arxiv.tex
\documentclass{amsart}
\title{Equivariant K-homology of Bianchi groups with non-trivial class group}
\author{Mathias Fuchs}

\address{Department of Bioinformatics, Center of Informatics, Statistics and Epidemiology, University Medical Center G\"ottingen}
\email{mfu@bioinf.med.uni-goettingen.de}
\urladdr{http://www.bioinf.med.uni-goettingen.de/people/mathias\_fuchs/}

\usepackage[latin1]{inputenc}
\usepackage[T1]{fontenc}
\usepackage{amsmath,amssymb,latexsym,array}
\usepackage{mathrsfs,nicefrac}
\usepackage{oldgerm}
\usepackage{amsxtra}
\usepackage[all]{xy}
\usepackage[short-journals,initials,numeric,lite]{amsrefs}
\usepackage{amsthm}
\usepackage{float}
\usepackage[small,bf]{caption}
\usepackage{hyperref}

\usepackage{mathptmx}

\renewcommand{\MR}[1]{{MR#1}}
\BibSpec{article}{
  +{}{\PrintAuthors} {author}
  +{,}{ } {title}
  +{,}{ } {journal}
  +{}{ \textbf} {volume}
  +{}{ \parenthesize} {date}
  +{,}{ } {pages}
  +{.}{ } {transition}
  +{.}{ } {review2}
}
\BibSpec{book}{%
  +{}{\PrintAuthors}{author}
  +{}{\PrintEditorsA}{editor}
  +{,}{ \textit}{title}
  +{:}{ \textit}{subtitle}
  +{,}{ }{type}
  +{,}{}{booktitle}
  +{,}{ \PrintEdition}{edition}
  +{,}{ }{series}
  +{,}{ vol.~}{volume}
  +{,}{ part~}{part}
  +{,}{ }{publisher}
  +{,}{ }{organization}
  +{,}{ }{place}
  +{,}{ }{date}
  +{}{ \parenthesize}{status}
  +{,}{ }{note}
  +{.}{ }{transition}
  +{.}{ }{review2}
}
\BibSpec{incollection}{%
  +{}{\PrintAuthors}{author}
  +{,}{ }{title}
  +{:}{ \textit}{subtitle}
  +{,}{ }{type}
  +{,}{ in \textit }{booktitle}
  +{,}{ \PrintEdition}{edition}
  +{}{ \PrintEditorsB}{editor}
  +{,}{ }{pages} 
  +{,}{ Proc.\textit }{conference}
  +{,}{ }{series}
  +{,}{ vol.~}{volume}
  +{,}{ part~}{part}
  +{,}{ }{publisher}
  +{,}{ }{organization}
  +{,}{ }{place}
  +{,}{ }{date}
  +{,}{ ISBN }{isbn}
  +{}{ \parenthesize}{status}
  +{,}{ }{note}
  +{.}{ }{transition}
  +{.}{ }{review2}
}
\BibSpec{manual}{%
  +{,}{ } {title}
  +{,}{ } {organization}
  +{,}{ } {address}
  +{,}{ \parenthesize } {date}
  +{,}{ } {note}
  +{,}{ } {title}
}
\BibSpec{thesis}{%
  +{}{\PrintAuthors}{author}
  +{}{\PrintEditorsA}{editor}
  +{}{\textit}{label}
  +{.}{ \textit}{title}
  +{,}{ \PrintThesisType}{type}
  +{,}{ part~}{part}
  +{,}{ }{organization}
  +{,}{ }{institution}
  +{,}{ }{place}
  +{}{ \parenthesize}{status}
  +{,}{ }{note}
  +{}{ }{date}
  +{.}{}{transition}
}

\newcommand{\Hy}{\mathcal{H}}

\newcommand{\C}{{\mathbb{C}}}
\newcommand{\R}{{\mathbb{R}}}
\newcommand{\Z}{{\mathbb{Z}}}
\newcommand{\N}{{\mathbb{N}}}
\newcommand{\Q}{{\mathbb{Q}}}

\newcommand{\rationals}{{\mathbb{Q}}}

\newcommand*{\Homol}{\operatorname{H}}
\newcommand*{\SL}{\operatorname{SL}}
\newcommand*{\PSL}{\operatorname{PSL}}
\newcommand*{\PSO}{\operatorname{PSO}}
\newcommand*{\PSU}{\operatorname{PSU}}

\newcommand*{\SO}{\operatorname{SO}}
\newcommand*{\SU}{\operatorname{SU}}
\newcommand*{\KK}{\operatorname{KK}}
\newcommand*{\K}{\operatorname{K}}
\newcommand*{\lcross}{\ltimes}
\newcommand*{\LHSnaught}{R\K^\Gamma_0(\underline{E}\Gamma)}
\newcommand*{\LHSone}{R\K^\Gamma_1(\underline{E}\Gamma)}
\newcommand*{\LHSstar}{R\K^\Gamma_*(\underline{E}\Gamma)}
\newcommand*{\Cred}{C^\ast_{\textit{red}}}
\newcommand*{\Cmax}{C^\ast_{\textit{max}}}

\renewcommand{\le}{\leqslant}
\renewcommand{\ge}{\geqslant}

\newcommand{\bs}{\backslash}

\newcommand{\Fin}{{\mathfrak F \mathfrak i \mathfrak n }}
\newcommand{\abs}[1]{\left|#1\right|}

\theoremstyle{plain}
\newtheorem{thm}{\bfseries Theorem}
\newtheorem{Lem}[thm]{\bfseries Lemma}
\newtheorem{PropDef}[thm]{\bfseries Proposition and Definition}

\newtheorem{cor}[thm]{\bfseries Corollary}

\newtheorem{df}[thm]{\bfseries Definition}
\newtheorem{fact}[thm]{\bfseries Fact}
\theoremstyle{remark}
\newtheorem{ex}[thm]{\bfseries Example}

\newtheorem{obs}[thm]{\bfseries Observation}

\newtheorem{rem}[thm]{\bfseries Remark}

\begin{document}
\begin{abstract}
We compute the equivariant K-homology of the groups PSL$_2$ of imaginary quadratic integers with trivial and non-trivial class-group. This was done before only for cases of trivial class number.\\
We rely on reduction theory in the form of the $\Gamma$-CW-complex defined by Fl\"oge. We show that the difficulty arising from the non-proper action of $\Gamma$ on this complex can be overcome by considering a natural short exact sequence of $C^\ast$-algebras associated to the universal cover of the Borel-Serre compactification of the locally symmetric space associated to $\Gamma$. We use rather elementary $C^\ast$-algebraic techniques including a slightly modified Atiyah-Hirzebruch spectral sequence as well as several 6-term sequences.\\
This computes the K-theory of the reduced and full group $C^\ast$-algebras of the Bianchi groups.
\end{abstract}

\maketitle 

\setcounter{secnumdepth}{4}
\setcounter{tocdepth}{3}

\section{Introduction}
This paper is concerned with the Bianchi groups, i.e. the class of groups defined by PSL$_2$ of the integers of an imaginary quadratic number field $\rationals [\sqrt{-m}]$ where $m$ is a product of different primes. Any such group is a group of orientation-preserving isometries of hyperbolic three space $\Hy$ by means of the right-coset identification $\Hy \cong \PSL(2, \C) / \PSU(2)$. We aim at computing the left hand side of the Baum-Connes-conjecture for these groups. Our main result is theorem \ref{statement} that accomplishes the calculation in the example case $m=5$.\\
There is vast literature on the Bianchi groups \cites{Fine1989,Elstrodt1998a,Maclachlan2003} and the conjecture \cite{Mislin2003}. The latter is concerned with the difficult problem to parametrize K-theory classes of projective modules over the reduced $C^\ast$-algebra of $\Gamma$. This $C^\ast$-algebra is defined as the closure in the operator norm of the left regular representation of $\Gamma$ on $\ell^2\Gamma$. The conjecture asserts that the K-theory of $\Cred\Gamma$ is isomorphic to equivariant K-homology with $\Gamma$-compact supports. We shall use the ``official'' definition
$$
R\KK_*(\mathscr{C}_0(\underline{E}\Gamma), \C) := \lim_{\substack{X\subset \underline{E}\Gamma:\\ \Gamma\backslash X \text{compact}}} \KK^\Gamma_*(\mathscr{C}_0(X), \C)
$$
of equivariant K-theory with $\Gamma$-compact supports as in \cite{Baum1994}, where $\underline{E}\Gamma$ is a classifying space for proper actions of $\Gamma$, and  $\KK^\Gamma$ is equivariant bivariant K-theory as defined  in \cite{Kasparov1988}. In the present context of a closed subgroup $\Gamma$ of a semisimple Lie group $G$, the symmetric space $G/K = \Hy$ is a typical universal space $\underline{E}\Gamma$ \cite{Baum1994}, where $K$ is a maximal compact subgroup of $G$.\\
The Baum-Connes conjecture applies equally to Lie groups and has, in fact, been proved for all connected Lie groups \cite{Chabert2003}. However, most often the Plan\-che\-rel formula allows to understand the structure of the left regular representation of Lie groups, whereas the $C^\ast$-algebras of arithmetic groups have turned out to be very hard to understand. One of the reasons is that for many arithmetic groups, the $C^\ast$-algebra is not type I, i.e. there are irreducible representations not contained in the operator ideal of compact operators.\\
Hence, the knowledge of the K-theory of the $C^\ast$-algebra yields valuable information which seems not to be accessible by more elementary or more explicit approaches. Furthermore, the Baum-Connes conjecture is, so far, the method of choice for computing the K-theory whenever the structure of the $C^\ast$-algebra is unknown.\\
However, it often turns out that the left-hand side is not immediately calculable either. The present paper is in spirit quite close to \cite{Sanchez2008} which computes the left-hand-side for the group $\Gamma = \SL(3,\Z)$; however, it is unknown if the assembly map is surjective in this case. In contrast, Julg and Kasparov have verified the Baum-Connes conjecture for all discrete subgroups of $\SO(n, 1)$ and $\SU(n, 1)$ \cite{Julg 1995}. Since $\PSL(2, \C) \cong \PSO(3, 1)$, this readily implies that the assembly map is an isomorphism for all Bianchi groups. Therefore, the Bianchi groups are arguably the most interesting arithmetic groups for which complete knowledge of (the isomorphism type of) $\mbox{K}_*(\Cred\Gamma)$ is available.\\
The canonical way to understand arithmetic groups is via actions on retracts of the symmetric space, as computed by reduction theory. Here, we shall use (a suitable subset of) Fl\"oge's CW-complex \cite{Floge1983}, a union of a 2-dimensional retract of hyperbolic 3-space with a countable subset of the Satake (spherical) boundary of $\Hy$. These boundary points of the complex are the $\Gamma$-orbits of so-called singular vertices of the Fl\"oge complex, and such orbits are in bijection with the non-trivial elements of the class group \cite{Serre1970}. The singular points are clearly visible in the visualizations of several exemplary fundamental domains in \cite{Floge1983} and \cite{Rahm2011a}. Contractibility of the Fl\"oge complex has been shown by Rahm and the author in \cite{Rahm2011a}.\\
The equivariant $K$-homology of Bianchi groups has been calculated by Rahm in \cite{Rahm2011}. However, he considers only the cases of trivial class number $m = 1,\allowbreak 2,\allowbreak 3,\allowbreak 7,\allowbreak 11,\allowbreak 19,\allowbreak 43,\allowbreak 67,\allowbreak 163$. The reason for this restriction is that in these cases there are no singular points. However, for higher class numbers the Fl\"oge complex is not proper anymore, as the stabilizers of the singular points are parabolic. An algorithmic approach for computation of the complex has been introduced, described and exploited for the computation of the integral homology by A.~Rahm and the author in \cite{Rahm2011a}. Furthermore, a program for the calculation of fundamental domains in the Pari/GP language has been published as part of Rahm's PhD thesis \cite{Rahm2010a}.\\
Since the presence of singular points in the Fl\"oge complex leads to the action being non-proper, in these cases the Fl\"oge complex is not a model for the universal classifying space $\underline{E}\Gamma$ of proper actions, and therefore not immediately suitable for calculation of equivariant $K$-homology for higher class numbers. Moreover, the Fl\"oge complex is not locally finite since there are infinitely many edges emanating from the singular points. So, there is no obvious locally compact topology on it, leading to severe difficulties in defining an associated commutative $C^\ast$-algebra.\\
There is a general construction for turning arbitrary, possibly non-proper, $\Gamma$-CW-complexes into proper actions \cite{Mislin2001} which could in principle be applied to the Fl\"oge complex; however, the construction is not in general cofinite and therefore difficult to use for computational purposes.\\
In the present paper, we show how to overcome these difficulties. The approach pursued in this paper is more akin of classical $C^\ast$-algebraic techniques, and hence closer to the original object under study, the $C^\ast$-algebra $\Cred\Gamma$. The four main ingredients are: reduction theory in the form of the Fl\"oge complex, the existence of the Borel-Serre compactification, a certain amount of Kasparov theory, and topology in the form of an Atiyah-Hirzebruch spectral sequence.\\
Let us give a short summary of the steps and results. Section \ref{Fl} introduces the topological objects under study. We identify the universal cover of the Borel-Serre compactification of the locally symmetric space associated to $\Gamma$ as a universal $\Gamma$-space (lemma \ref{isUP}). One may note that for torsion-free subgroups of Bianchi groups, Borel-Serre compactifications of the quotient $\Gamma\backslash \Hy$ (which is then a manifold) have already been mentioned in \cite{Grunewald2008}. There is an extension \ref{55} of $C^\ast$-algebras naturally associated to the inclusion of the boundary component into this universal space. We use this to show that the equivariant $K$-homology of $\Gamma$ forms a 6-term exact sequence (sequence \ref{Hy2LHS}) with topological $K$-homology of a disjoint sum of tori and the Kasparov K-homology of the crossed product $\Gamma \lcross \Hy$. The latter, in turn, is determined (lemma \ref{IsZero}) by the K-homology of the subset $X_\circ:= X\cap \Hy$ of the Fl\"oge complex $X$. It is necessary to work with $X_\circ$ instead of $X$ in order to overcome the problem that $X$ is not locally finite. The space $X_\circ$ is naturally endowed with a locally compact topology, so there is no problem in associating a well-behaved commutative $C^\ast$-algebra to it. The goal of section \ref{AHS} is to introduce the spectral sequence used to compute the K-homology of $X_\circ$. In principle, the Atiyah-Hirzebruch spectral sequence starting from Bredon homology associated to a group action on a complex is the classical tool to compute its equivariant K-homology. That spectral sequence is explained in \cite{Mislin2003} and was used in \cites{Sanchez2008,Rahm2010a,Rahm2011}. Here, we have to consider a slightly more general form of the spectral sequence (lemma \ref{Bigone}) in order to be able to treat the space $X_\circ$ instead of $X$. This generalization is done by identifying the Atiyah-Hirzebruch spectral sequence for the equivariant K-homology of a complex as the Schochet spectral sequence associated to a filtration by closed ideals of a $C^\ast$-algebra. In this case, these ideals come from the intersections of the skeleta of $X$ with $\Hy$. At that point, all technical difficulties are resolved, and we can illustrate the computation by means of the concrete example $m=5$ in section \ref{concrete}. This leads to the main theorem \ref{statement}.\\
The Bianchi groups, although far from being amenable, are K-amenable \cite{Cuntz1983}, as they are closed subgroups of a K-amenable group, see \cite{Kasparov1984} for the proof for $\SO(3, 1)$. As a consequence, there is a KK-equivalence between the reduced and full group $C^\ast$-algebra. Therefore, by calculating the K-theory of the former we simultaneously calculate that of the latter. Hence, the Bianchi groups may also be the most interesting lattices for which complete knowledge of the K-theory of the full algebra is now available.\\
I would like to thank Prof. Ralf Meyer very much for several inspiring discussions on the subject, which alone made this paper possible. I would also like to thank Prof. Edgar Wingender for the opportunity to complete the present work under his hood at the Medical Center G\"ottingen.
\section{Fl\"oge's complex, equivariant K-homology and the Borel-Serre boundary}\label{Fl}
\subsection{Fl\"oge's complex}
Briefly, the Fl\"oge complex $X^\bullet$ is defined as follows, see \cites{Floge1983,Rahm2011a} for details and references. Let now $m\in\N$ be square-free, and let $R=\mathcal{O}_{\rationals[\sqrt{-m}\thinspace]}$ and $\Gamma=\PSL(2, R)$. Denote the hyperbolic three-space by $\Hy=\C\times\R^*_+$. The space $\Hy$ is a homogeneous space under the Lie group $G = \PSL(2, \C)$. The Satake (or spherical) boundary $b\Hy$ of $\Hy$ is $\C P^1$, and the action of $G$, and hence that of $\Gamma$, extend continuously to actions on the boundary by M\"obius transformations \cite{Ratcliffe1994}*{Section 12}. A point $D\in \C P^1 -\{\infty\}$ is called a \emph{singular point} if for all $c,d\in R, c\ne 0, Rc+Rd=R$, we have $\left|cD-d\right|\ge 1$.\\
We shall denote the union of all orbits of singular points by $S$. Note that the point $\infty\in \C P^1$ is never in the orbit of a singular point. The set $S$ can either be viewed as a subset of $\C\times \{0\}\subset \C\times \R_{\ge 0}$, or of $\C P^1 - \infty$, or of $b\Hy - \infty$. In the first case, $S$ is identified with the number field $\mathbb{Q}[\sqrt{-m}]$ \cite{Siegel1965}.\\
The set $S$ is a subset of the set of cusps of $\Gamma$ as defined in \cite{Serre1970}. Serre considers the set $P$ of rational boundary points defined as $\Q[\sqrt{-m}]\cup \{\infty\}$ as a subset of the boundary viewed as the projective space $\mathbb{P}^1(\C)$. The action of $\Gamma$ falls into orbits which are in natural bijection with the class group, and the orbit corresponding to the trivial element is the orbit of $\infty$, whereas the set $S$ is the union of all orbits corresponding to non-trivial elements of the class group.\\
Furthermore, one considers the union of all hemispheres
$$
\left\{(z,r): \quad \left|z-\dfrac{\lambda}{\mu}\right|^2+r^2=\dfrac{1}{\left|\mu\right|^2}\right\} \subset \Hy,
$$
for any two $\mu,\lambda$ with $ R\mu+ R\lambda = R$, as well as the ``space above the hemispheres''
$$
B:=\bigl\{(z,r):  \left|cz-d\right|^2+r^2\left|c\right|^2\ge 1\\
 \text{ for all }c,d\in R,c\ne 0\text{ such that }Rc+Rd=R\bigr\}.
$$
Let $\abs{\cdot}$ denote the modulus in $\C$.
\begin{PropDef}[\cite{Floge1983}]
As a set, $\widehat{\Hy}\subset \C\times \R^{\ge 0}$ is the union $\Hy \cup S$ (where we omit, for simplicity, the standard identification $\C\times \{0\} \cong \C P^1 - \infty$.) The topology of $\widehat{\Hy}$ is generated by the topology of $\Hy$ together with the following neighborhoods of the translates $D\in S$ of singular points:
$$
\widehat{U}_\epsilon(D):=
\{D\} \cup \biggl\{ (z, r) \in\Hy : \abs{z - D}^2 + \biggl(r - \frac{\epsilon}{2} \biggr)^2 < \frac{\epsilon^2}{4}\biggr\}
$$
This definition makes $S$ a closed subset of $\widehat{\Hy}$.\\
There is a retraction $\rho$ from $\widehat{\Hy}$ onto the set $X := S\cup \Gamma \cdot \partial B \subset \widehat{\Hy}$ of the union $S$ with all $\Gamma$-translates of $\partial B$, i.~e. there is a continuous map $\rho:\widehat{\Hy}\to X$ such that $\rho(p)=p$ for all $p\in X$. The set $X$ admits a natural structure as a cellular complex $X^\bullet$, such that $\Gamma$ acts cellularly on $X^\bullet$. We shall refer to the complex thus defined, as well as to its structure as a $\Gamma$-subset of $\widehat{\Hy}$, as to the ``Fl\"oge complex''.
\end{PropDef}
We shall not make use of any topology on $X$, only of the structure as a subset of $\widehat{\Hy}$ and of the combinatorial structure. Since $B$ comprises Bianchi's fundamental polyhedron (which is called $D$ in \cite{Floge1983}), we have $\Gamma \cdot B = \Hy$ and hence this definition of $X$ coincides with Fl\"oge's original definition as the $\Gamma$-closure of $B\cup \{\text{singular points}\}$. Note that the topology of $\widehat{\Hy}$ is not the one inherited from the Satake (spherical) compactification of $\Hy$. However, $\widehat{\Hy}$ coincides on $\Hy$ with the usual topology; furthermore, it is path-connected, locally path-connected and simply connected \cite{Floge1983}*{Satz 1}, and contractible \cite{Rahm2011a}*{Lemma 8}, as is the cellular complex $X^\bullet$ \cite{Rahm2011a}*{Corollary 7}.\\
Since $X$ is not locally finite at the points of $S$, we shall not work with $X$ directly, but only with the pruned intersection $X\cap \Hy = X - S$ of $X$ with $\Hy$. On $X-S$, the cellular topology coincides with the topology inherited from $\Hy$, and is locally compact (unlike that of $\widehat{\Hy}$), whence there are no difficulties in associating a $C^\ast$-algebra $\mathscr{C}_0(X-S)$. Although $X-S$ is not a complex anymore, the $C^\ast$-algebra still possesses a filtration by closed ideals, defined by functions that vanish on the intersections of the skeleta of $X$ with $\Hy$. The associated Schochet-spectral sequence will be one of the keys to the calculation of the equivariant K-homology of $\Gamma$. In the following, the pruned cellular complex $X-S = X\cap\Hy$ endowed with the subset topology of $\Hy$ shall be denoted by $X_\circ$, and its pruned skeleta by $X_\circ^q := X^q \cap X_\circ$.\label{defofcirc}\\
It is important to note that any two edges adjacent to a common singular point in $X$ define half-open intervals that are \emph{disjoint} from each other as subsets of $X_\circ$. Furthermore, the space $X_\circ$ is, as subset of $\Hy$, locally compact. This contrasts the failure of local finiteness of $X$. As a warning, one should not expect to have $X_\circ$ the same equivariant K-homology as the subcomplex of $X$ that consists of all cells that do not touch a singular point although these spaces are equivariantly homotopic. The reason is that these homotopies are not proper since the singular points are at infinity, for the geometry of $X_\circ$. (Recall that the ``compactly supported'' version $\K^*(\mathscr{C}_0 (-))$ of K-homology is only invariant under proper homotopies.)
\subsection{The Borel-Serre boundary} Besides the Fl\"oge construction $\widehat{\Hy}$, we shall need another enlargment of $\Hy$, the Borel-Serre construction $\Hy(P)$ for the set $P = S \cup \Gamma \cdot \{\infty\}$ of rational boundary points. Very roughly speaking, the space $\Hy(P)$ is obtained from $\Hy$ by gluing copies of $\R^2$ onto each $D\in P$ (whereas $\widehat{\Hy}$ was obtained from $\Hy$ by merely gluing copies of a point onto each $D \in S$). Let us recall the setup and notation of Serre's introduction of the Borel-Serre boundary for linear algebraic groups of $\R$-rank one \cite{Serre1970}*{appendix 1}.\\
Serre's notation translates as follows. The space $\Hy$ is hyperbolic space, $b\Hy$ is the ordinary spherical (Satake) boundary of $\Hy$, $D$ is a translate of a singular point or of the point $\infty$, the subgroup $Q_D$ is its stabilizer inside $G$. We are going to view $\Hy$ as a space acted upon by $G$ from the right; thus, there is the natural identification $\Hy\cong K\bs G$. The group $Q_D$ is a minimal parabolic subgroup, and the singular points defined by Fl\"oge naturally embed into the Satake boundary. For instance, consider the simplest choice where $D$ is the origin of the Poincar\'e plane. Then $Q_D \cap \Gamma$ is the set of upper triangular matrices in $\Gamma$, hence isomorphic to $R=\mathcal{O}_{\rationals[\sqrt{-m}\thinspace]}$. Even though there are several orbits of singular points, all stabilizer groups $Q_D \cap \Gamma$ are free abelian of rank two \cite{Rahm2011a}.\\
The group $N_D$ is conjugated inside $G=\PSL(2, \C)$ to the group $\bigl(\begin{smallmatrix}1&\C\\ 0&1\end{smallmatrix}\bigr)$ whence $Q_D / N_D$ is isomorphic to $\C^*$, the multiplicative group of non-zero complex numbers. The space $\Hy_D$ is defined as the union of hyperbolic space with a boundary component $Y_D$, defined as the space of rank one-tori of $Q_D$. In our cases, all $Y_D$ are diffeomorphic to $\R^2$. More specifically, we obtain for $D=\infty$, the regular cusp at infinity of the Riemann sphere:
\begin{align*}
Q_\infty &=
\bigl\{
\bigl(
\begin{smallmatrix}
z&*\\
0&z
\end{smallmatrix}
\bigr)
:z\in\C
\bigr\}\\
T_{\PSU(2), \infty} &:=
\bigl\{
\bigl(
\begin{smallmatrix}
t&0\\
0&1/t
\end{smallmatrix}
\bigr)
:t\in\R^*_+
\bigr\}\\
Y_\infty &=
\bigl\{
\bigl(
\begin{smallmatrix}
1&z\\
0&1
\end{smallmatrix}
\bigr)
:z\in\C
\bigr\}
\end{align*}
Furthermore, there is a unique fixed point $S_K$ of the Cartan involution on $Y_D$ defined by $K$. This defines an Iwasawa decomposition $G = K\cdot A_K \cdot N_D$ associated to the subgroup $K$, where $A_K$ is the neutral connected component of $S_K$, and $.\cdot .$ denotes the ordinary multiplication in the group $G$. The group $A_K$ identifies with $\R_+^\ast$ in a canonical way (induced by the positive root of $S_K$ with respect to the order relation coming from $N_D$), and for $t\in\R_+^\ast$, the associated element in $A_K$ is denoted by $t_K$. These observations permit to topologize the space $\Hy_D$ by means of the family of bijective maps $f_K: \R_+ \times N_D\to \Hy_D$, defined by
$$
f_K(t, n) =
\begin{cases}
(K)\cdot t_K \cdot n\in K\bs G \cong \Hy & t > 0,\thinspace n \in N_D,\\
n^{-1} \cdot S_K\cdot n\in Y_D & t = 0, \thinspace n \in N_D.
\end{cases}
$$
where $K$ is a maximal compact subgroup of $G$. These maps are compatible for different choices of $K$ and therefore define a unique structure of manifold with boundary on $\Hy_D$, independent of $K$.\\
Let now $P$ be the set of cusps of $\Gamma$, i.e. the set of orbits of singular points together with the orbits of the trivial cusp at $\infty$. In Serre's terminology, the space $\Hy(P)$ is the union
$$
\Hy(P) =  \Hy \cup\bigcup_{D \in P} Y_D.
$$
defined by gluing the spaces $\Hy_D$ together along their common intersection $\Hy$. Hence, as a set, $\Hy(P)$ is the disjoint union $\Hy \cup (P \times \R^2)$. Serre shows that $\Hy(P) /\Gamma$ is compact, and called the Borel-Serre compactification of $\Hy/\Gamma$. The space $\Hy(P)$ is its universal cover.\\
Of course, it is also possible to consider a slightly smaller $\Gamma$-space, defined by considering only the subset $S$ of $P$,
$$
\Hy(S) = \Hy \cup\bigcup_{D \in S} Y_D\subset \Hy(P).
$$
in order to obtain a space that surjects onto $\widehat{\Hy}$. In fact, one has
\begin{obs}
There is a natural map $\Hy(S)\to \widehat{\Hy}$ defined by collapsing each boundary component to the corresponding singular point. This map is continuous and surjective.
\end{obs}
This is proved by checking directly that preimages of the open sets that define the topology on $\widehat{\Hy}$ are open in $\Hy(S)$. The relation between the spaces $\Hy, \widehat{\Hy}, \Hy(S)$ and $\partial\Hy(S)$ can be summarized by the equivariant diagram
$$
 \xymatrix@C=3em{
  \Hy(S)\ar[dr]&\partial\Hy(S)\ar[l]\ar[d]\\
  \Hy \ar[r]\ar[u]& \widehat{\Hy}.
 }
$$
The following lemma uses the notion of amenable transformation group, discussed in detail in \cite{Anantharaman2001}.
\begin{Lem}\label{biglemma}
Let $C = \mathscr{C}_0(Y)$ with $Y \in \{\Hy, \partial \Hy(P), \Hy(P), \partial \Hy(S), P\}$. Then the transformation group $(Y, \Gamma)$ is amenable. In particular, the $C^\ast$-algebraic crossed product $\Gamma\lcross C$ is nuclear and unique, i.e. the quotient map from the maximal to the reduced crossed product is an isomorphism.\\
\end{Lem}
The set $P$ can be replaced by any $\Gamma$-closed subset of $P$, but we shall not need that fact.
\begin{proof}
By \cite{Anantharaman2001}*{Theorem 5.3}, it suffices to show amenability of the transformation group.\\
For $C = \mathscr{C}_0(\Hy)$, write the symmetric space as a coset space $\Hy = G/K$ where $K$ is a maximal compact subgroup of $G$. For two closed subgroups $\Gamma$ and $K$ of a locally compact group $G$, the associated transformation group $(G/K, \Gamma)$ is an amenable transformation group \cite{Anantharaman2001}*{Example 2.7(5)}.\\
For $C = \mathscr{C}_0(\partial\Hy(P))$, observe that the action of $\Gamma$ on $\partial\Hy(P)$ is proper in the sense that the map
$$
\Gamma \times \partial \Hy(P) \to \partial\Hy(P) \times \partial\Hy(P)
$$
defined by $(\gamma, x)\mapsto (\gamma x,x)$ is topologically proper because the action is free and cocompact. Any proper action defines an amenable transformation group. In fact, the functions $g_i(x,\gamma) = h(\gamma^{-1}x)$ satisfy the conditions of \cite{Anantharaman2001}*{Propositions 2.2(2)} where $h$ is a continuous non-negative function on $\partial\Hy(P)$ such that $\sum_{\gamma\in\Gamma} h(\gamma^{-1}x) = 1$ for all $x\in \partial \Hy(P)$.\\
Form the associated extension of maximal crossed products associated to the invariant ideal $\mathscr{C}_0(\Hy)$ in $\mathscr{C}_0(\Hy(P))$ (for maximal crossed products, exactness is automatic). We have shown that the ideal and the quotient are nuclear (and therefore coincide with minimal crossed product).  Nuclearity for $Y = \Hy(P)$ then follows from the fact that nuclearity is stable under extensions \cite{Pedersen1979}. Moreover, since $\Gamma$ is discrete, it has property $(W)$ \cite{Anantharaman2001}*{Example 4.4(3)}, and therefore also the transformation groups $(\partial\Hy(P), \Gamma)$ and $(\Hy(P), \Gamma)$ are amenable \cite{Anantharaman2001}*{Theorem 5.8}.\\
For $C = \mathscr{C}_0(P)$, write $P$ as the disjoint union of orbits $P = \cup P_i$ (this is a decomposition in finitely, namely $m$ components, where $m$ is the class number). Any crossed product splits into a direct sum over these components, so it is enough to prove the statement separately for each. On such a component, we can apply \cite{Anantharaman2001}*{Example 2.7(5)} again, this time to the ambient group $\Gamma$ (which is locally compact) and the two closed subgroups $\Gamma$ itself and the stabilizer $\Gamma_i$, since the orbit $S_i$ is then equal to the $\Gamma$-space $\Gamma / \Gamma_i$, and $\Gamma_i\cong\Z^2$ is abelian, hence amenable.
\end{proof}
\subsection{A classifying space}
Lemma \ref{isUP} below uses the following simple criterion, well-known in the literature, for a space to be universal, i.~e. to serve as a model for the classifying space for proper actions for $\Gamma$. Note that whereas $\Hy$ is both a universal $\Gamma$- and $G$-space, the space $\Hy(P)$ is not acted upon by $G$.
\begin{fact}[\cite{Emerson2006}*{after Definition 3}] Any free and proper $\Gamma$-space is universal if and only if it is $H$-equivariantly contractible for any finite subgroup $H\subset\Gamma$.
\end{fact}
\begin{Lem}\label{isUP}
For any $\Gamma$-closed subset $S\subset P$, the universal cover $\Hy(P)$, as constructed in \cite{Serre1970}*{appendix 1}, of the Borel-Serre compactification of $\Gamma\backslash\Hy$ is a universal proper $\Gamma$-space.
\end{Lem}
\begin{proof}
By construction, the action of $\Gamma$ on $\partial \Hy(P)$ is free. Moreover, for each boundary component we can choose an isomorphism of its (non-pointwise) stabilizer with $\Z^2$ such that its action on the boundary component is equivariantly homeomorphic to the action of $\Z^2$ on $\R^2$, hence proper. Thus, the action on $\partial\Hy(P)$ is proper. It ensues that $\Hy(P)$ is proper.  Let $H$ be a finite subgroup of $\Gamma$. The quotient $H\backslash \Hy(P)$ is an orbifold with boundary which has the same homotopy type as its interior $H\backslash \Hy$. The latter is contractible because $\Hy$ is a universal proper $\Gamma$-space \cite{Baum1994}*{section 2}. So $H\backslash \Hy(P)$ is contractible. Moreover, since $\Hy(P)$ is contractible, the contraction of $H\backslash \Hy(P)$ thus obtained lifts to an $H$-equivariant contraction of $\Hy(P)$ whence the assertion.
\end{proof}
In fact, the space $\Hy$ itself is the typical example of a universal proper $\Gamma$-space, so $\Hy$ and $\Hy(P)$ can both be used as models for $\underline{E}\Gamma$ and for computation of equivariant K-homology. An isomorphism $R\K^\Gamma_*(\Hy) \to R\K^\Gamma_*(\Hy(P))$ is induced by the inclusion map $\Hy \to \Hy(P)$.
However, $\Hy(P)$ has, by construction, the advantage that it is cocompact unlike $\Hy$. Therefore, we can pass from $R\KK$ to $\KK$ as follows.
\begin{equation}
\begin{split}
R\KK^\Gamma(\mathscr{C}_0(\underline{E}\Gamma), \C) &\cong R\KK^\Gamma(\mathscr{C}_0(\Hy(P)), \C)\\
&\cong \KK^\Gamma(\mathscr{C}_0(\Hy(P)), \C)\\
&\cong \K^*(\Gamma\lcross \mathscr{C}_0(\Hy(P))).
\end{split}
\end{equation}
The last isomorphism is the dual Green-Julg theorem \cite{Blackadar1998}*{Theorem 20.2.7(b)}. Since the crossed product is nuclear, every ideal is semisplit \cite{Blackadar1998}*{Theorem 15.8.3}.\\
Throughout the paper, the reader should bear in mind that the notations $\KK^*(-,-)$ and $\K^*(-)$ refer to the ``original'' Kasparov KK-groups instead of the compact-support group $R\KK$. (We shall prefer to work with the former.) For a commutative $C^\ast$-algebra $\mathscr{C}_0(X)$, the group $\K^*(\mathscr{C}_0(X))$ is K-homology with locally finite support, rather than the usual group with compact support. Locally finite K-homology is called K-homology with compact support in \cite{Mislin2003}.
\subsection{$C^\ast$-extensions}
We can now make use of 6-term sequences which are available in $\KK$, in contrast to $R\KK$. The space $X_\circ$ was defined in subsection \ref{defofcirc}. The sequences
\vspace{-1em}
\begin{equation}\label{55}
\text{
\begin{minipage}{.6\textwidth}
$$
0 \to \mathscr{C}_0(\Hy - X_\circ) \to \mathscr{C}_0(\Hy) \to \mathscr{C}_0(X_\circ) \to 0
$$
$$
0 \to \mathscr{C}_0(\Hy) \to \mathscr{C}_0(\Hy(P)) \to \mathscr{C}_0(\partial \Hy(P)) \to 0,
$$
\end{minipage}
}
\end{equation}
defined by the respective evaluation maps, are exact and of course equivariant. Note that in the latter sequence we choose to work with $\Hy(P)$ instead of $\Hy(S)$ because only $\Hy(P)$ is cocompact, whence leading to equivariant K-homology.\\
Note that none of these algebras have a unit. We can take reduced crossed products by $\Gamma$. (We have shown that they coincide with the maximal ones.) It is well-known \cite{Williams2007}*{Rem. 7. 14} that a $\Gamma$-invariant ideal $I$ in a coefficient $\Gamma$-algebra $C$ induces a short exact sequence of reduced crossed products $0 \to \Gamma\lcross I \Gamma\lcross C\to \Gamma \lcross C/I \to 0$ when $\Gamma$ is discrete. Therefore, we arrive at a short exact sequence of reduced crossed products with the sequences \ref{55}. These induce exact 6-term sequences in K-homology $\KK(-, \C).$ For this, note that all algebras occurring in extensions throughout this paper are nuclear, and therefore all extensions are semisplit, so excision in KK-theory holds.\\
For a countable discrete group $\Gamma$, there is an identification of $\KK^\Gamma(A, C)$ with $\KK(\Gamma\lcross A, \C)$ \cite{Blackadar1998}.\\
The K-homology of the quotient $C^\ast$-algebras is easy to compute. To start with the second extension, the action of $\Gamma$ on $\partial \Hy(P)$ is free and proper, so there is a strong Morita equivalence $\Gamma\lcross \mathscr{C}_0(\partial \Hy(P)) \sim C(\Gamma\bs \partial\Hy(P))$. Let $k$ denote the class number of the underlying number field. Then the number of orbits of singular points is $k$. Hence, the quotient space is a disjoint union of $k$ compact 2-tori \cite{Serre1970}, so we have
\begin{equation}\label{tori}
\K^*(\Gamma\lcross \mathscr{C}_0(\partial \Hy(P))) \cong 
\begin{cases}
\mathbb{Z}^{2k}, & * = 0\\
\mathbb{Z}^{2k}, & * = 1.
\end{cases}
\end{equation}
Hence, the second ses. of \ref{55} allows to compute the equivariant K-homology of to compute that of $\Hy(P)$ from $\Hy$.\\
The former ses. of \ref{55}, in turn, allows to compute that of $\Hy$ from that of $X_\circ$. As we have shown in lemma \ref{isUP}, the space $\Hy(P)$ is a classifying space, so this will achieve the computation.\\
The interior of Bianchi's fundamental polyhedron has trivial stabilizer \cite{Rahm2011a}. Therefore, the action on $\Hy - X_\circ$ is free and proper so $\Gamma\lcross \mathscr{C}_0(\Hy - X_\circ)$ is Morita equivalent to $\mathscr{C}_0(\Gamma \bs (\Hy - X_\circ))$ which is isomorphic to $\mathscr{C}_0$ of an open 3-cell, i.e. $\mathscr{C}_0(\R^3)$. Therefore,
$$
\KK_*(\Gamma\lcross \mathscr{C}_0(\Hy - X_\circ), \C) \cong
\begin{cases}
0, &* = 0\\
\Z, &* = 1.
\end{cases}
$$
Thus, we can compute the left hand side $\K^*(\Gamma\lcross \mathscr{C}_0(\Hy(S)))$ via successive computation of the groups $\K^*(\Gamma\lcross \mathscr{C}_0(X_\circ))$ and $\K^*(\Gamma\lcross \mathscr{C}_0(\Hy))$ without support condition, by the exact 6-term sequences
\begin{equation}\label{X2Hy}
  \xymatrix{
    \K^0(\Gamma\lcross \mathscr{C}_0(X_\circ)) \ar[r] & \K^0(\Gamma\lcross \mathscr{C}_0(\Hy)) \ar[r] & 0 \ar[d]\\
    \Z \ar[u] & \K^1(\Gamma\lcross \mathscr{C}_0(\Hy)) \ar[l] & \K^1(\Gamma\lcross \mathscr{C}_0(X_\circ)) \ar[l]
  }
\end{equation}
and, using \ref{tori},
\begin{equation}\label{Hy2LHS}
  \xymatrix{
    \mathbb{Z}^{2k} \ar[r] & \LHSnaught \ar[r] & K^0(\Gamma\lcross \mathscr{C}_0(\Hy)) \ar[d]\\
    K^1(\Gamma\lcross \mathscr{C}_0(\Hy))\ar[u] & \LHSone \ar[l] & \mathbb{Z}^{2k} \ar[l]
  }
\end{equation}
which paves the way to reduce the computation of the equivariant K-homology of $\Gamma$ to that of the $K$-homology of the reduced crossed product and that of the boundary tori.
\begin{rem}
The invertible in $\KK^\Gamma_1(\C, \mathscr{C}_0(\Hy))$ (namely, Kasparov's dual-Dirac element) leads to an isomorphism $K^*(\Gamma\lcross \mathscr{C}_0(\Hy))\cong \K^{*+1}(\Cmax\Gamma) = R_{*+1}(\Gamma)$. Therefore, rewriting the 6-term sequence \ref{Hy2LHS} and using the fact that the assembly map is an isomorphism, we arrive at the following exact 6-term-sequence.
\begin{equation}\label{54}
  \xymatrix{
    \mathbb{Z}^{2k} \ar[r] & \K_0(\Cred\Gamma)\ar[r] & R_1(\Gamma) \ar[d]\\
    R_0(\Gamma)\ar[u] & \K_1(\Cred\Gamma) \ar[l] & \mathbb{Z}^{2k} \ar[l]
  }
\end{equation}
where $R(\Gamma) = \KK^\Gamma(\C, \C)$ is the Kasparov representation ring. This is remarkable since there are rarely exact sequences available that connect K-theory and K-homology of the same algebra (in view of $R(\Gamma) = \KK(\Cmax\Gamma, \C) = \K^*(\Cred\Gamma)$ by K-amenability.)
\end{rem}
\begin{Lem}\label{IsZero}
The connecting homomorphism of the 6-term-sequence \ref{X2Hy} is zero, so there is an isomorphism $\K^0(\Gamma\lcross \mathscr{C}_0(X_\circ))\cong \K^0(\Gamma\lcross \mathscr{C}_0(\Hy))$ and a short exact sequence
$$
0\to \K^1(\Gamma\lcross \mathscr{C}_0(X_\circ)) \to \K^1(\Gamma\lcross \mathscr{C}_0(\Hy)) \to \Z \to 0.
$$
\end{Lem}
\begin{proof}
Let $\mathcal{D}\subset \Hy$ denote the interior of the Bianchi fundamental polyhedron which is called $D$ in \cite{Floge1983} (in the present manuscript, the letter $D$ is already occupied); hence, there are homeomorphisms $\mathcal{D} = \R^3$ and $\Hy - X_\circ = \Gamma \times \mathcal{D}$ because the interior of $\mathcal{D}$ is trivially stabilized. It is well known that the KK-equivalence coming from the strong Morita-Rieffel equivalence
$$
\Gamma\lcross \mathscr{C}_0(\Hy - X_\circ) \sim \mathscr{C}_0(\mathcal{D})
$$
is induced by the inclusion of $C^\ast$-algebras $\mathscr{C}_0(\mathcal{D}) \to \Gamma \lcross \mathscr{C}_0(\Hy - X_\circ)$ where the first algebra is viewed as a crossed product with the trivial group, and a function on $\mathcal{D}$ is viewed as a function on $\Hy - X_\circ$ by setting it to zero outside $\mathcal{D}$.\\
Consider the map of short exact sequences
\begin{equation}
\xymatrix{
0 \ar[r]& \mathscr{C}_0(\mathcal{D})\ar[d] \ar[r]^{=} & \mathscr{C}_0(\mathcal{D}) \ar[r]\ar[d]  & 0 \ar[r]\ar[d] & 0\\
0 \ar[r]& \Gamma\lcross \mathscr{C}_0(\Hy - X_\circ) \ar[r]& \Gamma\lcross \mathscr{C}_0(\Hy) \ar[r] & \Gamma\lcross \mathscr{C}_0(X_\circ) \ar[r] &0
  }
\end{equation}
We claim that the center vertical arrow induces a surjective map
$$
K^1(\Gamma\lcross \mathscr{C}_0(\Hy)) \to \K_1(\R^3)\cong\Z
$$
in K-homology (recall that we are dealing with locally finite homology). The assertion then follows.\\
Consider the following diagram of $\KK(\C, -)$-theory groups,
$$
\xymatrix{
 & \K^1(\mathcal{D})\ar[d]\\
\K^0(\text{pt}) \ar[ur]\ar[r] \ar[d]& \K^1(\Hy)\ar[d]\\
\K_0(\Cred\Gamma)\ar[r] & \K_1(\Gamma\lcross\Hy)
}
$$
where all vertical arrows are induced by $C^ \ast$-inclusions, the diagonal arrow is the Bott isomorphism associated to the standard orientation of $\R^3$, the upper horizontal arrow is Bott isomorphism, and the lower horizontal arrow is multiplication with Kasparov's dual-Dirac element $\beta \in \KK_1^\Gamma(\C, \mathscr{C}_0(\Hy))$. The commutativity of the triangle is inherent in the definition of the Bott element, and that of the square follows from the fact that the forgetful homomorphism $\KK^\Gamma(\C, \mathscr{C}_0(\Hy)) \to \KK(\C,\Hy)$ sends $\beta$ to the Bott element.\\
All arrows except the two lower vertical arrows are isomorphisms (recall that $\beta$ is invertible). It follows that the dual map $ \K^1(\R^3) \to K_1(\Gamma\lcross \mathscr{C}_0(\Hy))$ is injective, since the class of the unit in $\K_0(\Cred\Gamma)$ is non-zero and non-torsion. It is classical that this follows from the existence of the trace map $\K_0(\Cred\Gamma)) \to \Z$ which sends $[1]$ to $1$.  This completes the proof.\\
\end{proof}
\section{The Atiyah-Hirzebruch spectral sequence}\label{AHS}
\subsection{Bredon homology}\label{Bredon}
Bredon homology is the main tool for computation of equivariant K-homology in similar contexts \cites{Sanchez2008,Rahm2010a,Rahm2011}, so we refer to these references for a more thorough introduction. Here, we shall only fix the notation.\\
Recall the usual terminology for Bredon homology: The orbit category $\mathcal{O}_\Gamma$ of $\Gamma$ has an object $\Gamma / H$ for each subgroup $H$ of $\Gamma$, and all $\Gamma$-maps between any two objects as morphisms. A Bredon module is a functor $\mathcal{O}_\Gamma \to \textit{Ab}$ from the orbit category to abelian groups. Here, we are going to consider the Bredon module $M_q$ that associates to $\Gamma / H$ the K-homology group $\K^q(\C H)$. This reduces to even degree where it is $\text{R}(H)$ canonically. Let $H\to H'$ be a homomorphism of discrete countable groups. The differentials $\K^q(\Cmax H) \to \K^q(\Cmax H')$ of the Atiyah-Hirzebruch spectral sequence are given by left multiplication by an abstractly defined element $d\in\KK_0(\Cmax H', \Cmax H)$. Bredon homology of a proper complex $Y$ with respect to the family $\Fin$ of finite subgroups with coefficients in the Bredon module $M$ is defined as the homology of the cellular complex with coefficients in $M$
$$
\Homol_p(Y; M) = \Homol_p(C(Y) \otimes_\Fin M) := \Homol_p\bigoplus_{\sigma\in \Gamma\backslash Y^\bullet} M\Gamma_\sigma,
$$
where the sum extending over orbit representatives. In case $H$ and $H'$ are finite, $M_q$ is a functor $\Gamma /H \mapsto RH$, and $d$ defines a map $RH\to RH'$ which reduces to the usual Bredon differential as described in \cite{Sanchez2008}, and in even more detail in \cite{Rahm2010a}. Specifically, the latter reference gives the information about the morphisms induced on the representation rings by the finite group inclusions occurring among stabilizers of Fl\"oge simplices, as well as details about how to simplify the computations in the Pari/GP language~\cite{PARI2}. In fact, the complex representation ring of a finite group as the free $\Z$-module the basis of which are the irreducible characters of the group. These irreducible characters are given by the character tables for the finite subgroups of the Bianchi groups. One has to consider all possible inclusions, and identify the said morphism; then, this information is fed into the program Bianchi.gp in order to obtain the Bredon chain complex, from which we shall easily deduce the information on the modified Bredon chain complex, introduced below.
\subsection{The Schochet spectral sequence for a vertex-pruned cellular complex}
Lemma \ref{IsZero} motivates to calculate the K-homology of $\Gamma\lcross \mathscr{C}_0(X_\circ)$. In the usual setting of a proper cellular complex, there is the Atiyah-Hirzebruch spectral sequence computing equivariant K-homology. Its $E^1$-term is the Bredon complex, and its $E^2$-term is Bredon homology. For the present purpose, we modify this setting in the following way. Note that $X_\circ$ is not a cellular complex. However, there is still a natural notion of $p$-cell, with the difference that there are one-cells that are adjacent to only one vertex. In the cellular complex $X$, there are one- and two-cells $\sigma$ adjacent to singular points. These are compact subsets of $X$. However, their intersection $\sigma\cap \Hy \subset X_\circ$ is, of course, not compact in $X_\circ$. However, this makes no difference with regard to the combinatorial structure: They are still adjacent to the same edges. Since the definition of Bredon homology only takes into account the combinatorial structure, we may still write down a combinatorial Bredon complex for $X_\circ$, and this complex still satisfies $d^2 = 0$.
\begin{df}
We shall refer to this complex as the \emph{modified Bredon complex} and accordingly to its homology as the \emph{modified Bredon homology}, denoted by $\Homol_p(X_\circ; M_q)$ with coefficients in the Bredon module $M_q$.
\end{df}
The following lemma says, roughly speaking, that this modification of Bredon homology is not only a modification of the $E^1$ term, but that in fact this modified $E^1$-term fits into an entire modified spectral sequence whose $E^1$-term is the modified $E^1$-term. Furthermore, this modified spectral sequence computes ``the right thing'', namely the equivariant K-homology of $X_\circ$.\\
Let  $M_q$ denote the Bredon module that associates to $\Gamma / H$ the K-homology group $M_qH := \K^q(\C H)$ as above.
\begin{Lem}\label{Bigone}
There is a homological spectral sequence computing equivariant K-ho\-mo\-lo\-gy,
$$
E^1_{p, q} = \bigoplus_{\sigma \in \Gamma\backslash X^p_\circ} M_q\Gamma_\sigma \implies \K^{p+q}(\Gamma\lcross \mathscr{C}_0(X_\circ)),
$$
where the sum extends over representatives $x$ of orbits of $p$-cells in $X_\circ$, and $\Gamma_x$ is the respective stabilizer isomorphism type. It is concentrated in the first and fourth quadrants, and 2-periodic in the index $q$. Its differential $d^n$ has bidegree $(-n, n - 1)$.\\
The summand of the $E^1$-term belonging to the cells of $X_\circ$ that do not touch a singular point is equal to the usual Bredon complex of these cells. The summand of the $E^1$-term belonging to the one-cells of $X_\circ$ that touch a singular point has only a differential to one vertex orbit. The summand of the $E^1$-term belonging to the two-cells of $X_\circ$ that touch a singular point is equal to the usual Bredon complex associated to the combinatorial structure. Therefore, the $E^2$-term is $\Homol_p(X_\circ; M_q)$.\\
\end{Lem} 
\begin{proof} This is a slight generalization of the Atiyah-Hirzebruch spectral sequence. It is constructed as the spectral sequence associated as in \cite{Schochet1981} to the three-step filtration by closed ideals
$$
0\subset \Gamma\lcross \{f: f| X_\circ^1=0 )\subset \Gamma\lcross \{f: f| X_\circ^0=0 )\subset \Gamma\lcross \mathscr{C}_0(X_\circ)
$$
of functions vanishing on the pruned skeleta.\\
 The sub-quotient
$$
(\Gamma\lcross \{f: f| X_\circ^{q-1}=0 )) /(\Gamma\lcross \{f: f| X_\circ^{q}=0 )) = \Gamma\lcross \mathscr{C}_0(X_\circ^q - X_\circ^{q-1})
$$
is a direct sum over the orbits of $q$-cells, and a summand corresponding to the orbit of a cell $c$ stabilized by the (necessarily finite) subgroup $\Gamma_c\subset \Gamma$ is isomorphic to $\mathscr{C}_0(\R^q)\otimes (\Gamma\lcross \mathscr{C}_0(\Gamma c))$ which is strongly Morita equivalent to the $q$-fold suspension of the complex group ring $\C\Gamma_c$ whose K-homology is isomorphic to $R\Gamma_c$. This yields the Atiyah-Hirzebruch spectral sequence for equivariant K-homology, upon implementing the $q$-fold dimension shifts pertaining to the cells' dimensions. Its $E^1$ computes (and, therefore, its $E^2$-term is) Bredon homology with respect to finite subgroups \cite{Mislin2003}, as it is considered in \cite{Sanchez2008} for computation of equivariant K-homology.
\end{proof}
The fact that we deal with pruned skeleta introduces no peculiarities since on the level of sub-quotients it only means to pay attention to the fact that for each edge $e$ touching a singular point, there is only one inclusion of vertex stabilizers instead of two (Recall that each edge touches at most one singular point).
\begin{ex}
It is easy to visualize the modified Atiyah-Hirzebruch spectral sequence thus defined for a very simple toy case of a pruned cellular complex, for instance the case of a single edge acted upon by the trivial group. One can then compare with the Schochet spectral sequence associated to a half-open interval with the filtration $0\subset \mathscr{C}_0(\R^*_+) \subset \mathscr{C}_0(\R_{\ge 0})$, where the endpoint is the single simplex. The latter's $E^1$-term takes the form $\Z \leftarrow \Z$ in even rows, and zero in odd rows, according to a single edge and a single point. The homology of this vanishes, in accordance with $\K^*(\mathscr{C}_0(\R_{\ge 0})) =0$. Bearing this example in mind might help in understanding the slight generalization of the Atiyah-Hirzebruch spectral sequence to pruned cellular complexes, considered here. Roughly speaking, ``it is enough to omit the missing points from the spectral sequence''.
\end{ex}
\subsection{Short exact sequences}
Let us write down how to pass, in our case where the spectral sequence is applied to the algebra $\Gamma\lcross \mathscr{C}_0(X_\circ)$, from $E^\infty$ to the K-homology using the edge homomorphisms.
\begin{cor}
The odd rows of the $E_2$-term vanish since $K_1$ of a finite-dimensional $C^\ast$-algebra such as a finite group's ring necessarily vanishes. Since $\text{dim}\ X_\circ = 2$, the spectral sequence is concentrated in the columns $0 \le p \le2$. Therefore, $E^2 = E^\infty$, and there is an isomorphism
\begin{equation}\label{K1}
E^2_{1, 0} \cong \K^1(\Gamma\lcross \mathscr{C}_0(X_\circ))
\end{equation}
and a short exact sequence
\begin{equation}\label{extension}
0 \to E^2_{0, 2} \to \K^0(\Gamma\lcross \mathscr{C}_0(X_\circ)) \to E^2_{2, 0} \to 0. 
\end{equation}
\end{cor}
\section{Explicit calculations}\label{concrete}
\begin{table}
\tiny
\makebox[\textwidth][c]{
$
\begin{array}{c|rrrrrrrrrrrrr}
& (b, a) & (b, a) & (v, v_1) & (v, v_1) & (v, v_1) & (a_3, u) & (a_3, u) & (u, b) & (u, b) & (u_1, b) & (u_1, b) & (a, v) & (a, s) \\
\hline
 \text{large cell} & -1 & -1 & -1 & -1 & -1 & -1 & -1 &  &  & -1 & -1 &  & \\
 \text{mid-size cell} & 1 & 1 &  &  &  & 1 & 1 & 1 & 1 &  &  &  & \\
 \text{small cell} &  &  &  &  &  &  &  &  &  &  &  &  & 
\end{array}
$
}
\normalsize
\caption{The transpose of the differential $d^1_{2, 0}$ for $m = 5$, with zeroes omitted. There are multiple column names because they were chosen to indicate only the originating cell instead of including the full character information, in order to save space. Rank is two, elementary divisors are zero and one. The small cell is in the kernel of $d^1_{2, 0}$.}
\label{d_one}
\end{table}
\begin{table}
\tiny
\makebox[\textwidth][c]{
$
\begin{array}{c|rrrrrrrrrrrrr}
& (b, a) & (b, a) & (v, v_1) & (v, v_1) & (v, v_1) & (a_3, u) & (a_3, u) & (u, b) & (u, b) & (u_1, b) & (u_1, b) & (a, v) & (a, s) \\
\\
\hline
b & -1 &   &   &   &   &   &   & 1 &   & 1 &   &   &   \\
\vline &   & -1 &   &   &   &   &   &   & 1 & 1 &   &   &   \\
\vline & -1 &   &   &   &   &   &   &   & 1 &   & 1 &   &   \\
b &   & -1 &   &   &   &   &   & 1 &   &   & 1 &   &   \\
u &   &   &   &   &   & 1 &   & -1 &   & -1 &   &   &   \\
\vline &   &   &   &   &   &   & 1 &   & -1 & -1 &   &   &   \\
\vline &   &   &   &   &   &   & 1 & -1 &   &   & -1 &   &   \\
u &   &   &   &   &   & 1 &   &   & -1 &   & -1 &   &   \\
a & 1 &   &   &   &   & -1 &   &   &   &   &   & -1 & -1 \\
a &   & 1 &   &   &   &   & -1 &   &   &   &   & -1 & -1 \\
v &   &   &   &   &   &   &   &   &   &   &   & 1 &   \\
\vline &   &   &   &   &   &   &   &   &   &   &   & 1 &   \\
v &   &   &   &   &   &   &   &   &   &   &   & 1 &   \\
\end{array}
$
}
\normalsize
\caption{The differential $d^1_{1, 0}$ for $m = 5$. Row and columns information is shortened as in Table \ref{d_one}, whence the occurrence of multiple row and column names. The elementary divisor one occurs with multiplicity $7$, and elementary divisor two occurs with multiplicity one.}
\label{d_two}
\end{table}
Let us go through all computations for $m= 5$. Details, including a picture, of a simple fundamental domain for the action of $X$ are given in \cite{Floge1983} and \cite{Rahm2011a}*{section 3.2}. We shall stick to the notations therein.\\
Recall that we set out to calculate the spectral sequence, associated to a commutative filtrated algebra, that almost comes from the skeleton filtration of a cellular complex' algebra. However, there is ``one vertex orbit missing'' in $X_\circ$. So, there are only four orbits of vertices, namely those of $a, b, u, v$, but the same numbers of edges and faces as in \cite{Rahm2011a}*{section 3.2}, namely seven orbits of edges, those of $ba, vv_1, a_3u, ub, u_1b, av, as$, and three orbits of faces. Stabilizers are given in explicit form in the same references. There is only one orbit of singular points, since the class number of $\mathbb{Q}[\sqrt{-5}]$ is two.\\
All information on the representation rings $\K^0(\Cmax \Gamma_\sigma) = R\Gamma_\sigma$ of finite stabilizers of vertices $\sigma$ is given in \cite{Rahm2010a} and used as explained in subsection~\ref{Bredon}.\\
The $E^1$-term is a complex
$$
d^1: \Z^{4 + 4 + 2 + 3} \xleftarrow{d^1_{1, 0}} \Z^{2 + 3 + 2 + 2 + 2 + 1 + 1} \xleftarrow{d^1_{2, 0}} \Z^{1+1+1}
$$
in even rows, and zero in odd rows. The $d^1$-matrices are displayed in Tables \ref{d_one} and \ref{d_two}. Their elementary divisors are easily determined (for instance by the Pari/GP software \cite{PARI2}) by calculating the Smith normal forms, yielding
$$
d^1_{1, 0} \sim \text{diag}(0, 0, 0, 0, 0, 2, 1, 1, 1, 1, 1, 1, 1)
$$
 and
$$
d^1_{2, 0} \sim \text{diag}(0, 0, 0, 0, 0, 0, 0, 0, 0, 0, 0, 1, 1).
$$
Therefore, the modified Bredon homology is 
\begin{equation}
\begin{split}
E^2_{0, 0} & \cong \Z^5 \oplus \Z/2\\
E^2_{0, 1} &\cong \Z^3\\
E^2_{0, 2} &\cong \Z,
\end{split}
\end{equation}
where we write $\Z/g := \Z/g\Z$. (A quick check of a necessary criterion is the correct rational Euler characteristic $13 - 13 +3 = 5 -3 +1$.)
Using \ref{K1} and \ref{extension}, we obtain
$$
\K^*(\Gamma\lcross \mathscr{C}_0(X_\circ)) =
\begin{cases}
\Z^6 \oplus \Z/2 , &* = 0\\
\Z^3, &* = 1.
\end{cases}
$$
Using lemma \ref{IsZero}, we have
$$
\K^*(\Gamma\lcross \mathscr{C}_0(\Hy)) =
\begin{cases}
\Z^6 \oplus \Z/2 , &* = 0\\
\Z^4, &* = 1.
\end{cases}
$$
We can now state the main theorem.
\begin{thm} \label{statement}
 The equivariant $K$-homology of the Bianchi group PSL$_2(\mathcal{O}(\rationals [\sqrt{-5}]))$ is
$$
\LHSstar \cong
\begin{cases}
\Z^6 \oplus \Z/2 , &* = 0\\
\Z^4, &* = 1.
\end{cases}
$$
\end{thm}
\begin{proof}
  It only remains to solve the 6-term extension problem \ref{Hy2LHS}. Recall that the class number $k$ is two. The strategy for solving this 6-term problem is similar to the one used in \cite{Rahm2012} where the analogous long term sequence in homology was computed.
\end{proof}
As stated in the introduction, this simultaneously computes $\K_*(\Cred\Gamma)$ and $\K_*(\Cmax\Gamma)$.\\
Analogous computations for other $m$ should present no fundamentally new difficulties. However, they are beyond the scope of the present paper.\\
Another possibly interesting question asks for the least matrix size over $\Cred\Gamma$ containing an idempotent that realizes any given K-theory class associated to a list of generators of the left hand side thus computed.
\input{main2.bbl}

\end{document}